\documentclass{amsart}
\usepackage{mathpazo}
\usepackage[scaled=.95]{helvet}
\usepackage{courier}

\advance \topmargin by -\headheight
\advance \topmargin by -\headsep
\evensidemargin \oddsidemargin
\theoremstyle{plain}
\newtheorem{thm}{Theorem}[section]
\newtheorem{prop}[thm]{Proposition}
\newtheorem{lemma}[thm]{Lemma}
\newtheorem{cor}[thm]{Corollary}
\newtheorem{conj}{Conjecture}
\newtheorem*{signconj}{Euler Characteristic Sign Conjecture}
\theoremstyle{remark}
\newtheorem{remark}{Remark}

\title{The Euler Characteristic of a Haken 4-Manifold}
\author{Allan L. Edmonds}
\address{Department of Mathematics, Indiana University, Bloomington, IN 47405}
\email{edmonds@indiana.edu}

\begin{document}
\maketitle
\begin{abstract}
Haken $n$-manifolds are  aspherical manifolds, defined and studied by B. Foozwell and H. Rubinstein, that can be successively cut open along essential codimension-one submanifolds until a disjoint union of $n$-cells is obtained. Such manifolds come equipped with a boundary pattern, a particular kind of decomposition of the boundary into codimension-zero submanifolds.
We prove that  there is a certain numerical function $\varphi(X^{4})$ depending only on the boundary and boundary pattern  of the compact Haken 4-manifold $X^{4}$ (and vanishing if $X^{4}$ has empty boundary),  such that for any  compact Haken 4-manifold $X^{4}$ the Euler characteristic satisfies the inequality  $\chi (X^{4})\geq \varphi(X^{4})$. In particular, if $X^{4}$ is a closed Haken $4$-manifold, then $\chi (X^{4})\geq 0$.
\end{abstract}

\section{Introduction}\label{sec:intro}

We address the following fundamental conjecture about aspherical manifolds for a class of manifolds known as \emph{Haken $n$-manifolds}.
\begin{signconj}
If $X^{d}$ is a closed, aspherical manifold of even dimension $d = 2m$, then the Euler characteristic of $X^{d}$ satisfies $(-1)^m \chi(X^{d}) \geq 0.$
\end{signconj}
\noindent
Recall that a manifold or cell complex $X$ is said to be \emph{aspherical} if $\pi_i(X) = 0$ for all $i \geq 2$ or, equivalently, the universal covering space of $X$ is contractible.  The conjectured sign corresponds to the sign of the Euler characteristic of a product of $m$ surfaces of genus $g \ge 1$. This conjecture was first proposed as a question by W. Thurston in the 1970s. (See the Kirby problem set  \cite{Kirby1997}.) The first interesting and in general still unresolved case is in dimension 4.

Here we study this problem for the more tractable class of aspherical manifolds known as Haken $n$-manifolds, introduced by B. Foozwell in his thesis \cite{Foozwell2007, Foozwell2011}  and developed in recent years by Foozwell and H. Rubinstein \cite{FoozwellRubinstein2011, FoozwellRubinstein2012}. Haken manifolds generalize the well-known examples in dimension 3 explored in particular by F. Waldhausen in the late 1960s. Loosely put, Haken manifolds are manifolds that can be reduced to a disjoint union of cells by successively cutting the manifold open along essential hypersurfaces in the manifold. They provide a convenient framework for proving theorems by induction on dimension, and, most important for us, by induction on the number of steps necessary to reduce to a disjoint union of cells. They also provide a way of injecting topological combinatorics into the study.

\begin{thm}
If $X^{4}$ is a compact Haken 4-manifold without boundary, then $\chi(X^{4})\ge 0$.
\end{thm}

It is a necessary challenge is to find an interpretation of the conjecture for Haken manifolds that applies to aspherical manifolds with boundary that is not necessarily aspherical or $\pi_{1}$-injective, because such manifolds arise along the way as one cuts open a Haken manifold.
In general, the boundary of a Haken $n$-manifold comes equipped with a suitable regular decomposition into manifold pieces, called a \emph{boundary pattern}, in practice induced by the process of cutting open a manifold along hypersurfaces. 

\begin{thm}
There is a  numerical function $\varphi(X^{4})$ depending only on the boundary and boundary pattern  of the compact Haken 4-manifold $X^{4}$ (and vanishing if $X^{4}$ has empty boundary),  such that for any  compact Haken 4-manifold $X^{4}$ the Euler characteristic satisfies the inequality  $\chi (X^{4})\geq \varphi(X^{4})$.
\end{thm}

A crucial step is the following result.

\begin{thm}\label{thm:boundaryflag}
The boundary complex of a Haken 4-cell is dual to a flag simplicial 3-sphere.
\end{thm}

Duality and the flag complex condition are discussed in Section \ref{sec:duality}.
Using this result the main theorem is connected to the 3-dimensional Charney-Davis conjecture \cite{CharneyDavis1995a}, proved by M. Davis and B. Okun \cite{DavisOkun2001}. Their work proves the Euler Characteristic Sign Conjecture for closed, non-positively curved, cubical 4-manifolds.

It turns out that the analog of Theorem \ref{thm:boundaryflag} is true in all dimensions. Details will be deferred to another paper.

\subsection*{Outline}
In Section \ref{sec:facts} we present an overview of Haken $n$-manifolds, including basic definitions, examples, and fundamental theorems. In Section \ref{sec:approach} we describe our approach to the Euler Characteristic Sign Conjecture for Haken manifolds. In Section \ref{sec:results} we carefully state three technical results and state and prove the main theorem from the more {}technical results. In Section \ref{sec:phifunction}    we derive the technical results about the form of a possible $\varphi$-function and a corresponding transformation law describing how it behaves when cutting open a manifold along a hypersurface. In Section \ref{sec:duality} we analyze more deeply the way the $\varphi$-function behaves for Haken cells and explain why the dual of a boundary complex of a Haken 4-cell gives a flag simplicial 3-sphere, thus relating our problem to the well-known Charney-Davis conjecture.

\subsection*{Acknowledgements}This paper grew out of the author's talk ``Introduction to Haken $n$-manifolds, applied to the Euler characteristic of an aspherical 4-manifold'', delivered to the Groups, Geometry and Dynamics Conference in 
Almora, India, in December 2012. The author thanks the conference organizers for the opportunity to develop these thoughts and to present them in the form of a more-detailed paper. The author also expresses his appreciation to Hyam Rubinstein from whom he first learned about Haken $n$-manifolds, to Bell Foozwell for helpful correspondence about Haken $n$-manifolds, and to Steve Klee for correspondence related to dualizing the condition of being a Haken $n$-cell. Thanks also to Mike Davis for discussion that showed that Theorem \ref{thm:boundaryflag} generalizes. We expect to include that extension together with a subsequent analysis that reduces the Euler Characteristic Sign Conjecture for Haken $n$-manifolds to the corresponding Charney-Davis conjecture in dimension $n-1$ in a forthcoming joint paper. Finally, we gratefully acknowledge the constructive comments of the referee, whose corrections and suggestions were much appreciated.

\section{Basic Facts about Haken $n$-Manifolds}\label{sec:facts}
This section is adapted from the articles of Foozwell and Rubinstein \cite{FoozwellRubinstein2011, FoozwellRubinstein2012}, to which we refer for more complete details. We offer an impressionistic description in which we emphasize points relevant to the present application and de-emphasize less relevant points.

\subsection{The idea of a Haken $n$-manifold}

Informally, a Haken $n$-manifold is a compact $n$-manifold with possible boundary together with a given prescription for cutting it along a succession of properly embedded, two-sided, codimension-one submanifolds so that in the end the result is a collection of $n$-balls.  The formal technical details of the definition will be summarized in several steps below. The development is due to Foozwell \cite{Foozwell2007} and is based upon Johannson's concept of a \emph{boundary pattern}, which the latter used in his study of properties of Haken 3-manifolds.

Of course  classical surfaces of non-positive Euler characteristic, and tori in all dimensions  have this property. We also have a well-developed theory in dimension three, which was initiated by Haken \cite{Haken1962}, brought to prominence by Waldhausen \cite{Waldhausen1968}, and studied from the present combinatorial point of view by Johannson \cite{Johannson1979, Johannson1994} .

\emph{Three simple examples to ponder:}  
\begin{enumerate}
\item
A $2$-cell, is aspherical with aspherical boundary but the boundary is not $\pi_{1}$-injective.  The same applies to $S^{1}\times D^{2}$.
\item
A nontrivial classical knot exterior gives an example of an aspherical manifold with aspherical and $\pi_{1}$-injective boundary.
\item
An $n$-ball, $n\ge 3$, is aspherical with non-aspherical boundary that is $\pi_{1}$-injective.
\end{enumerate} 

To deal with these situations we require that the boundary be decomposed into codimension-zero submanifolds called faces that are aspherical and $\pi_{1}$-injective. The same requirement applies to intersections of faces.

The term \emph{Haken $n$-manifold} will be defined in several steps. To begin with it is a compact manifold with possible boundary. The boundary is endowed with a  boundary pattern, that is, a finite collection of compact, connected $(n -1)$-manifolds in the boundary, such that the intersection of any $k$ of them is either empty or consists of $(n - k)$-manifolds, for $k = 1, \dots , n + 1$. 

The boundary pattern is required to be ``complete'' and ``useful''.

``Completeness'' means that the union of the faces is the entire boundary. All boundary patterns considered in this paper will be complete. If no particular boundary pattern is prescribed or implied, then it is understood that the boundary pattern consists of the components of the boundary.

``Usefulness'' is more technical but implies the following three conditions: 
\begin{enumerate}\label{useful}
\item
each face maps $\pi_{1}$-injectively into the manifold; 
\item
when the manifold is simply connected, the intersection of any two faces is connected; 
\item
when the manifold is simply connected, if three faces intersect pairwise nontrivially then  their three-fold intersection is nonempty.
\end{enumerate}

The full usefulness condition also implies that there are no interesting relations among the images of the fundamental groups of two adjacent or three mutually adjacent faces in the fundamental group of the manifold. We will not be much concerned with this aspect here. The full definition of usefulness is couched in terms of ``small 2-disks'' being trivial in an appropriate sense. More precisely, a proper map of a 2-disk into the manifold that meets the boundaries of faces transversely in 0, 2, or 3 points is required to be homotopic to a disk in the boundary such that the pre-image of the boundaries of faces is empty, a single arc, or a triod, respectively. See \cite{Foozwell2011, FoozwellRubinstein2011} for more details.

The components of the intersections of the faces give rise to a structure called a ``boundary complex'', analogous to a regular CW complex, except that the manifold pieces need not be cells in general.

\emph{Haken $n$-cells}, fundamental building blocks for general Haken $n$-manifolds, are defined inductively as follows. A point is a Haken 0-cell. An interval, with the boundary pattern consisting of the two boundary points, is a Haken 1-cell.  Usefulness requires that a Haken $2$-cell is a $p$-gon, with $p\ge 4$, i.e. with the boundary pattern given by the $p$ edges.

For $n\ge 3$ a Haken $n$-cell is an $n$-cell with complete and useful boundary pattern whose faces are themselves Haken $(n-1)$-cells. It follows that any $k$  faces intersect in a Haken $(n-k)$-cell, or have empty intersection.  

Important examples of Haken cells arise as convex polytopes.  Among the regular polytopes in dimension 3, for example, the cube and dodecahedron are Haken 3-cells, while the tetrahedron, octahedron, and icosahedron (which have triangular faces) are not. The natural cell-structure on the boundary of the octahedron does not even define a boundary pattern, since there are pairs of triangles that intersect in a single vertex rather than in an edge.

When one cuts open a manifold along a hypersurface meeting all the faces (and their faces, etc.) transversely, then one also cuts open the faces of the boundary pattern and creates a new boundary pattern of the cut-open manifold that includes two disjoint copies of the cutting hypersurface. The intention is to do this in such a way that the cut-open manifold is somehow simpler than the original.

Finally, a Haken $n$-manifold is an $n$-manifold $M$ together with a \emph{hierarchy,} that is, a given choice of a succession of manifolds with complete useful boundary patterns and hypersurfaces 
$$(M_{0},F_{0}), \dots, (M_{k},F_{k})$$ where $M_{0}$ is the original manifold with its given boundary pattern, and $M_{1}$ is the result of cutting open $M_{0}$ along $F_{0}$ and giving $M_{1}$ the induced boundary pattern, etc. We require that $M_{k+1}$ is a disjoint union of Haken $n$-cells. This completes the formal definition of Haken $n$-manifolds.

In dimension 3, Waldhausen \cite{Waldhausen1968} proved that an irreducible 3-manifold is Haken if and only if it contains a (two-sided) essential surface. In particular one can start with an essential surface, cut open along it, and then construct the rest of a hierarchy. There appears to be no such theorem in higher dimensions.

If the original Haken $n$-manifold has nonempty faces, then it follows from the definitions that those faces are also Haken $(n-1)$-manifolds, and so on through the faces of the full boundary complex.

See Foozwell-Rubinstein  \cite{FoozwellRubinstein2011} for more details.

\subsection{Some standard Haken manifolds}
Here we compile a brief list of Haken manifolds to be used in subsequent sections. Usually we do not explicitly describe a hierarchy. A surface of non-positive Euler characteristic, with boundary pattern consisting of the boundary components is always a Haken 2-manifold. One can add vertices and edges, at least 2 of each for each boundary component, and still have a Haken 2-manifold. A $p$-gon  is a Haken $2$-cell, provided $p\ge 4$. 

Any standard Haken $3$-manifold with $\pi_{1}$-injective boundary components is also a Haken 3-manifold in the present sense, with boundary pattern consisting of the components of the boundary. 

If $G^{3}$ is a Haken $3$-manifold without boundary, then $G^{3}\times S^{1}$ is a Haken $4$-manifold without boundary. Also $G^{3}\times I$ is a Haken $4$-manifold with boundary pattern consisting of  two copies of $G^{3}$.

If $G^{3}$ is a Haken 3-manifold with boundary a surface $T_{g}$ of genus $g\ge 1$, then $G^{3}\times S^{1}$ is a Haken 4-manifold with boundary $T_{g}\times S^{1}$ and a single boundary pattern face $T_{g}\times S^{1}$. And $G^{3}\times I$ is a Haken 4-manifold with boundary pattern consisting of $T_{g}\times I$ and $G^{3}\times \{0\}$ and $G^{3}\times\{1\}$.

If $T_{g}$ denotes a surface of genus $g\ge 1$, then $T_{g}\times I^{2}$ is a Haken 4-manifold with boundary pattern faces consisting of four copies of $T_{g}\times I$.

\subsection{Recent results about Haken $n$-manifolds}

Here are three  basic results about Haken manifolds that generalize results in dimension 3.
\begin{thm}[Foozwell, see \cite{FoozwellRubinstein2011}]
A Haken $n$-manifold is aspherical.
\end{thm}

And more substantially,
\begin{thm}[Foozwell \cite{Foozwell2011}]
The universal covering space of a Haken $n$-manifold is homeomorphic to $\mathbb{R}^{n}$.
\end{thm}

From this it follows that not all aspherical manifolds are Haken or ``virtually Haken'', since the universal covering space of certain aspherical manifolds first constructed by Mike Davis are not simply connected at infinity and hence not homeomorphic to euclidean space.

\begin{thm}[Foozwell \cite{Foozwell2007}]
The word problem for the fundamental group of a Haken $n$-manifold is solvable.
\end{thm}

Recently Foozwell and Rubinstein have proved 

\begin{thm}[Foozwell-Rubinstein \cite{FoozwellRubinstein2012}]
A closed Haken 3-manifold is the boundary of a Haken 4-manifold (with boundary pattern consisting of the components of the 3-manifold).
\end{thm} 

The most fundamental open problem is the following.

\begin{conj}
A homotopy equivalence between closed Haken $n$-manifolds is homotopic to a homeomorphism.
\end{conj}

\section{Approach to the Sign Conjecture for Haken 4-Manifolds}
\label{sec:approach}\label{sec:approach}
We establish here a framework suitable for proving the Euler characteristic inequality for Haken $4$-manifolds by induction on the length of a hierarchy. In particular we develop an appropriate version of the Euler characteristic conjecture for Haken manifolds with boundary.

We seek a function $\varphi$ that assigns to each Haken 4-manifold (and, indeed, any compact 4-manifold with a complete boundary pattern) a real number $\varphi(X^{4})$ depending only on the boundary of the manifold and its boundary pattern such that $\chi(X^{4})\ge \varphi(X^{4})$  for all compact  Haken 4-manifolds.  In particular, if $X^{4}$ is a Haken $4$-cell, then $\chi(X^{4})=1$ and we must require that $\varphi(X^{4})\le 1$ in that case.  We require that $\varphi(X^{4})=0$ when $\partial X^{4}=\emptyset$. In addition, if $G^{3}\subset X^{4}$ is an essential hypersurface in the hierarchy for $X^{4}$ and we set $Y^{4}=X^{4}|G^{3}$ ($X^{4}$ cut open along $G^{3}$. Formally $Y^{4}$ is obtained by removing a product neighborhood $N$ of $G^{3}$ of the form $G^{3}\times I$ from $X^{4}$.), then we require that
\[
\chi(X^{4})=\chi(Y^{4})-\chi(G^{3})\ge \varphi(Y^{4}) - \chi(G^{3}) \ge \varphi(X^{4})
\]
(The equality is a consequence of the standard sum theorem for Euler characteristics, and the first inequality would be a consequence of an induction hypothesis.)

If there is such a function $\varphi$ then an argument by induction on the length of a hierarchy implies that $\chi(X^{4})\ge \varphi(X^{4})$ for all Haken 4-manifolds.

Thus, in summary, we require
\begin{enumerate}
\item $\varphi(X^4)\le 1$ whenever $X^4$ is a Haken 4-cell.
\item $\varphi(X^4) =0 $ whenever $\partial X^4 = \emptyset$.
\item $\chi(X^{4})\ge \varphi(X^{4})$ whenever we can explicitly calculate both quantities.
\item $ \varphi(Y^{4}) - \chi(G^{3}) \ge \varphi(X^{4})$ whenever $G^{3}\subset X^{4}$ is an essential hypersurface in $X^{4}$ and $Y^{4}=X^{4}|G^{3}$.
\end{enumerate}

If $X^{n}$ is a Haken $n$-manifold, then $f_{n-1}(X^{n})$ denotes the number of facets, and, more generally, $f_{k}(X^{n})$ for $k< n$ denotes the number of $k$-dimensional elements of the complex of  intersections of the facets. 
We let $F^{k}<X^{4}$ denote a $k$-face of the boundary complex of $X^{4}$, $f_{k}(X^{4})$ denote the number of $k$-faces of the boundary complex of $X^{4}$, and $b_{k}(\partial X^{4})$ denote the $k$th betti number of the boundary, i.e., the dimension of $H_{k}(\partial X^{4};\mathbb{Q})$.

We will suppose that $\varphi$ takes the following  general form:
\[
\varphi(X^{4})=\sum_{k=0}^{3}r_{k}f_{k}(X^{4}) + \sum_{k=0}^{3}s_{k}\sum_{F^{k}<X^{4}}\chi(F^{k}) +\sum_{k=0}^{3}t_{k}b_{k}(\partial X^{4})
\] 
and then attempt to deduce restrictions on what the twelve coefficients might be.  The as yet unknown constant coefficients are not required to be integers.

\subsubsection*{Notation} When $X$ is a manifold with boundary pattern we will sometimes use $F^{k}X$ to denote the set of $k$-dimensional facets and write $\chi F^{k}X$  for $ \sum_{F^{k}< X}\chi(F^{k})$.

\section{Results}\label{sec:results}
We will prove the following general result about the possible $\varphi$-functions discussed in Section \ref{sec:approach}.

\begin{thm}\label{thm:possibleformula}
If there is a function of Haken 4-manifolds of the form
\[
\varphi(X^{4})=\sum_{k=0}^{3}r_{k}f_{k}(X^{4}) + \sum_{k=0}^{3}s_{k}\sum_{F^{k}<X^{4}}\chi(F^{k}) +\sum_{k=0}^{3}t_{k}b_{k}(\partial X^{4})
\] 
satisfying 
\begin{enumerate}
\item $\varphi(X^4)\le 1$ whenever $X^4$ is a Haken 4-cell.
\item $\chi(X^{4})\ge \varphi(X^{4})$ whenever we can explicitly calculate both quantities.
\item $ \varphi(Y^{4}) - \chi(G^{3}) \ge \varphi(X^{4})$ whenever $G^{3}\subset X^{4}$ is an essential hypersurface in $X^{4}$ and $Y^{4}=X^{4}|G^{3}$.
\end{enumerate}
then it can be expressed in the form
\[
\varphi(X^{4})=-\frac{1}{16}f_{0}(X^{4})  +\frac{1}{4} \sum_{F^{3}<X^{4}}\chi(F^{3}).
\] 
\end{thm}
After showing that this is the only possible $\varphi$-function, we derive two fundamental facts which will allow us to prove our main result.

\begin{prop}\label{prop:transformationlaw}
Let $\varphi$ be a function defined for Haken 4-manifolds by the formula 
\[
\varphi(X^{4})=-\frac{1}{16}f_{0}(X^{4}) + \frac{1}{4} \sum_{F^{3}<X^{4}}\chi(F^{3}).
\]
Let $X^{4}$ be a Haken $4$-manifold, let $G^{3}\subset X^{4}$ be a connected essential hypersurface, and let $Y^{4}=X^{4}| G^{3}$, with the induced boundary pattern. Then
\[
\varphi(Y^{4}) 
=  \varphi(X^{4}) -\frac18f_{0}(G^{3})  
 + \frac14\left(\chi (F^{2}G^{3})+\chi (\partial G^{3})  \right) .
\]
\end{prop}

\begin{prop}\label{prop:haken4cell}
Let $\varphi$ be a function defined for Haken 4-manifolds by the formula 
\[
\varphi(X^{4})=-\frac{1}{16}f_{0}(X^{4}) + \frac{1}{4} \sum_{F^{3}<X^{4}}\chi(F^{3}).
\]
If $X^{4}$ is a Haken 4-cell, then  $\varphi(X^{4})\le 1$.
\end{prop}

Here is the full statement of our positive result about the Euler Characteristic Sign Conjecture.
\begin{thm}\label{thm:main}
If $X^{4}$ is a Haken 4-manifold, then
\[
\chi(X^{4})\ge -\frac{1}{16}f_{0}(X^{4})  +\frac{1}{4} \sum_{F^{3}<X^{4}}\chi(F^{3}).
\] 
\end{thm}
\begin{proof}[Proof of Theorem \ref{thm:main} using Propositions \ref{prop:transformationlaw} and \ref{prop:haken4cell}]
We will prove this by induction on the length of a hierarchy. 
At the initial stage of the induction we need to verify the inequality for a Haken 4-cell.  That is we need
 $-\frac{1}{16}f_{0}(X^{4})+ \frac{1}{4} \sum_{F^{3}<X^{4}}\chi(F^{3})\le 1$.
But this is exactly the conclusion of Proposition \ref{prop:haken4cell}.

Second we need to consider the case where $Y^{4}=X^{4}|G^{3}$, where $G^{3}$ is the initial cutting hypersurface in the hierarchy of $X^{4}$.
We use induction on the length of a hierarchy, together with the earlier analysis to obtain
\begin{align*}
\chi(X^{4}) &= \chi(Y^{4}) - \chi(G^{3}) \text{ by the sum formula for the Euler characteristic}
\\
&\ge \varphi(Y^{4}) - \chi(G^{3}) \text{ by the inductive hypothesis}
\\
&=\varphi(X^{4})   -\frac{1}{8}f_{0}(G^{3})  
 + \frac{1}{4}\left(\chi (F^{2}G^{3})+\chi \partial (G^{3})  \right) -\chi (G^{3})
 \text{ by Proposition \ref{prop:transformationlaw}}
\\
&= \varphi(X^{4})   -\frac{1}{8}f_{0}(G^{3})  
 + \frac{1}{4}\left(\chi F^{2}(G^{3})+\chi \partial (G^{3})  \right) - \frac{1}{2}\partial \chi (G^{3}) \text{ since } \chi (G^{3}) = \frac{1}{2}\partial \chi (G^{3})
\\
 &= \varphi(X^{4}) 
 -\frac{1}{8}f_{0}(G^{3})   + \frac{1}{4}\left(\chi (F^{2}G^{3}) - \chi(\partial G^{3})   \right)
 \\
 &= \varphi(X^{4}) 
 -\frac{1}{8}f_{0}(G^{3})   + \frac{1}{4}\left(\chi (F^{2}G^{3}) - \chi (F^{0}G^{3}) + \chi (F^{1}G^{3}) - \chi (F^{2}G^{3})   \right)
\\
 &= \varphi(X^{4}) 
 -\frac{1}{8}f_{0}(G^{3})   + \frac{1}{4}\left(- \chi (F^{0}G^{3}) + \chi (F^{1}G^{3})  \right)
 \\
 &= \varphi(X^{4}) 
 -\frac{1}{8}f_{0}(G^{3})   + \frac{1}{4}\left(- f_{0}(G^{3}) + \frac{3}{2}f_{0}(G^{3})  \right) \text{ since } 3f_{0}(G^{3}) = 2\chi (F^{1}G^{3})
 \\
   &= \varphi(X^{4}) 
 + \left(-\frac{1}{8} + \frac{1}{8}\right)f_{0}(G^{3}) 
  \\
    &= \varphi(X^{4})
\end{align*}
\end{proof}

\begin{cor}
If $X^{4}$ is a compact Haken 4-manifold without boundary, then $\chi(X^{4})\ge 0$.
\end{cor}
\begin{proof}
By Theorem \ref{thm:main} we have $\chi(X^{4})\ge \varphi(X^{4})$, which vanishes since $\partial X^{4}=\emptyset$.
\end{proof}

It remains to prove Theorem \ref{thm:possibleformula} and Propositions \ref{prop:transformationlaw} and \ref{prop:haken4cell}, which will occupy the subsequent sections.

\section{Determination of $\varphi$-function and the transformation law}\label{sec:phifunction}

 \subsection{Some general reductions in the 4-dimensional case}
We consider a function of the form
\[
\varphi(X^{4})=\sum_{k=0}^{3}r_{k}f_{k}(X^{4}) + \sum_{k=0}^{3}s_{k}\sum_{F^{k}<X^{4}}\chi(F^{k}) +\sum_{k=0}^{3}t_{k}b_{k}(\partial X^{4})
\] 
satisfying the properties $(1) - (4)$ of Section \ref{sec:approach}.
First of all, note that we may (and shall)  assume that $t_{2}=t_{3}=0$, since $b_{0}=b_{3}$ and $b_{1}=b_{2}$ for any closed oriented 3-manifold.

\begin{lemma}
Suppose that $Z$ has the structure of a ``manifold complex'' (such as given by the boundary pattern of a Haken manifold). Then $\chi(Z)=\sum_{F<Z}(-1)^{\dim F}\chi(F)$.
\end{lemma}
\begin{proof}
By induction on the number of faces. The result is  clearly true for a single point. Suppose $Z=Y\cup G$, where $G\cap Y = \partial G$. By the sum formula for the Euler characteristic we have 
\begin{align*}
\chi(Z)&=\chi(Y)+\chi(G)-\chi(G\cap Y)\\
&=\chi(Y)+\chi(G)-\chi(\partial G)\\
&=\sum_{F<Y}(-1)^{\dim F}\chi(F)+\chi(G)-\chi(\partial G) \quad\text{ by induction.}
\end{align*}
If $\dim G$ is even, then $\chi(\partial G)=0$ and we obtain the desired formula for $X$. If $\dim G$ is odd, then $\chi(G)-\chi(\partial G)=-\chi(G)$ and we again obtain the desired formula. 
\end{proof}
\begin{cor}
If $\varphi$ corresponds to the 4-tuple $(s_{0},s_{1},s_{2},s_{3})$ and $\varphi'$ to the 4-tuple given by $s_{k}'=s_{k}+(-1)^{k}p, k=0,\dots, 3$, for  any number $p$ (and all other coefficients are the same), then $\varphi'=\varphi$.
\end{cor}
\begin{proof}
$\varphi'(X^{4})$ differs from $\varphi(X^{4})$ by a multiple of $\chi(\partial X^{4})$.
But the boundary of a compact 4-manifold always has Euler characteristic 0.
\end{proof}

\begin{cor}\label{cor:s2}
If $X^{4}$ is a Haken 4-manifold, then $$\sum_{k=0}^{3}(-1)^{k}\sum_{F^{k}<X^{4}}\chi(F^{k}) = 0.$$
\end{cor}
\begin{proof}
The sum represents the Euler characteristic of the 3-manifold $\partial X^{4}$ and hence vanishes.
\end{proof}

\begin{lemma}\label{lemma:s1}
If $X^{4}$ is a Haken 4-manifold, then $$2f_{0}(X^{4})=2\sum_{F^{0}<X^{4}}\chi(F^{0})=\sum_{F^{1}<X^{4}}\chi(F^{1}).$$
\end{lemma}

\begin{proof}
The quantity $\sum_{F^{1}<X^{4}}\chi(F^{1})$ counts the number of 1-faces that are not circles, i.e., are intervals, each with 2 end points. On the other hand every vertex has 4 edges incident at it, so $4\sum_{F^{0}<X^{4}}\chi(F^{0})$ counts the total number of edge ends. It follows that
\[
4\sum_{F^{0}<X^{4}}\chi(F^{0})=2\sum_{F^{1}<X^{4}}\chi(F^{1}).
\]
\end{proof}

\begin{cor}
In the function $\varphi$  we may, without loss of generality assume that $s_{1}=s_{2}=0$.
\end{cor}
\begin{proof}

By Lemma \ref{lemma:s1} we can replace the term $\sum_{F^{1}<X^{4}}\chi(F^{1})$ by $2f_{0}(X^{4})$ and the term $\sum_{F^{0}<X^{4}}\chi(F^{0})$ by $f_{0}(X^{4})$.
\end{proof}

\begin{lemma}
In the function $\varphi(X^{4})$ we may assume that $s_{0}=0$
\end{lemma}
\begin{proof}
The quantities $f_{0}(X^{4}) $ and $\sum_{F^{0}<X^{4}}\chi(F^{0})$ are equal and so may be combined into a single term.
\end{proof}

At this stage we may assume our function $\varphi(X^{4})$ has the following form
\[
\varphi(X^{4})=\sum_{k=0}^{3}r_{k}f_{k}(X^{4})+ s_{3}\sum_{F^{3}<X^{4}}\chi(F^{3}) +t_{0}b_{0}(\partial X^{4})+t_{1}b_{1}(\partial X^{4}).
\]

\subsection{Evaluation on some Standard Haken 4-Manifolds}
Here we continue the process of determining restrictions on the coefficients appearing in the formula for the $\varphi$-function by evaluating the formula on a variety of Haken 4-manifolds and cutting hypersurfaces in hierarchies.
\begin{prop}
The coefficients in the function $\varphi$ must satisfy 
\begin{align*}\label{prop:inequalities}
 t_{1}&=0\\
 t_{0}&=-r_{3}\\
r_{3}&=-r_{2}\\
 s_{3}&=1/4\\
 r_{2}&\le 0\\
 r_{3}&\ge 0\\
 r_{1}+r_{2}&\ge 0 \\
 r_{1}&\ge 0.
\end{align*}
\end{prop}
\begin{proof}
First consider $X^{4}=G^{3}\times S^{1}$, where $G^{3}$ is any closed, connected,  Haken 3-manifold, with first cutting surface in the hierarchy being $G^{3}\times \text{ point}$. Here $X^{4}|G^{3}\cong G^{3}\times I$, with boundary pattern consisting of the two copies of $G^{3}$. Then, interpreting the sequence of general inequalities 
\[
\chi(X^{4})=\chi(Y^{4})-\chi(G^{3})\ge \varphi(Y^{4}) - \chi(G^{3}) \ge \varphi(X^{4}),
\]
we obtain
\[
0=0-0\ge \varphi(G^{3}\times I)=2r_{3}+2t_{0}+2t_{1}b_{1}(G^{3})\ge \varphi(G^{3}\times S^{1})=0
\] 
and so $r_{3} + t_{0} + t_{1}b_{1}(G^{3})=0$. Since there are closed Haken 3-manifolds with arbitrary $b_{1}$, this forces $t_{1}=0$. And that forces $r_{3}+ t_{0} = 0$, or $t_{0}=-r_{3}.$

Next consider $X^{4}=G^{3}\times S^{1}$ where $G^{3}$ is a Haken 3-manifold with $\partial G^{3}=T_{g}$, a closed orientable surface of genus $g\ge 1$. Then $\partial X^{4}=T_{g}\times S^{1}$, with boundary pattern consisting of a single face. We have $Y^{4}=G^{3}\times I$ and $\partial Y^{4}=G^{3}\cup T_{g}\times I \cup G^{3}$. Again we interpret the sequence of inequalities
\[
\chi(X^{4}) = \chi(Y^{4})-\chi(G^{3}) \ge \varphi(Y^{4})-\chi(G^{3}) \ge \varphi(X^{4}).
\]

We have 
\begin{align*}
0&=(1-g)-(1-g)\\
 &\ge 2r_{2} + 3r_{3} + s_{3}(4(1-g)) -r_{3} - (1-g) \\
 &\ge r_{3}(1) + s_{3}(0) - r_{3}(1)\\
 &=0.
\end{align*}
Thus
\[
2r_{2}+2r_{3}+(1-g)(4s_{3}-1)=0.
\]
Letting $g=1$, we have $r_{2}+r_{3}=0$ and letting $g>1$ we then have $4s_{3}-1=0$, so that $s_{3}=\frac{1}{4}$.

Now consider $X^{4}=T_{g}\times I^{2}$, where $T_{g}$ is a smooth surface of genus $g\ge 1$.  Here $I^{2}$ denotes the 2-disk with its standard boundary pattern consisting of 4 edges and 4 vertices. The product inherits a boundary pattern and Haken manifold structure in a natural way, with four faces of the form $T_{g}\times I$. Then we get $$4r_{2}+4r_{3}+ 4s_{3}(2-2g)+t_{0}\le 2-2g.$$
Setting $g=1$, we find that $4r_{2}+4r_{3}+t_{0}\le 0$, hence $4r_{2}+3r_{3}\le 0$. Since we already knew that $r_{2}+r_{3}=0$, we see that $r_{2}\le 0$ and $r_{3}=-r_{2}\ge 0$.

We next consider $G^{3}\cong S^{1}\times I^{2}$, where $S^{1}$ is a non-separating simple closed curve in $T_{g}$, and we set  $Y^{4}=X^{4}|G^{3}= T_{g-1,2}\times I^{2}$, where $T_{g-1,2}$ denotes a twice-punctured surface of genus $g-1$.  Again we interpret the full sequence of inequalities
\[
\chi(X^{4}) = \chi(Y^{4})-\chi(G^{3}) \ge \varphi(Y^{4})-\chi(G^{3})\ge \varphi(X^{4}).
\]
Then identifying the various faces of $X^{4}$ we have $X^{4}$ has $F^{2}X^{4}=4(T_{g}\times\text{pt})$, and $F^{3}X^{4}=4(T_{g}\times I)$. Similarly $Y^{4}$ has $F^{1}Y^{4}=8(S^{1}\times\text{ pt})$, $F^{2}Y^{4}=8(S^{1}\times I)\cup 4(T_{g-1,2})$, and $F^{3}Y^{4}=4(T_{g-1,2}\times I)$.

Thus we obtain
\[
2-2g \ge \varphi (Y^{4})-\chi (G) = 8r_{1}+r_{2}(12-4+1)+\frac{1}{4}(4)(2-2g) - 0 \ge r_{2}+2-2g
\]
Thus
\[
0\ge 8r_{1}+9r_{2}\ge r_{2}
\]
which implies that $r_{1}+r_{2}\ge 0$. Since $r_{2}\le 0$ we have $r_{1}\ge 0$.
\end{proof}

At this stage we may assume our function $\varphi(X^{4})$ has the following form
\[
\varphi(X^{4})=r_{0}f_{0}(X^{4})+ r_{1}f_{1}(X^{4})+r_{2}\left(f_{2}(X^{4})-f_{3}(X^{4})+b_{0}(\partial X^{4})\right)+ \frac{1}{4}\sum_{F^{3}<X^{4}}\chi(F^{3}),
\]
where
\[
r_{1}\ge 0, r_{2}\le 0  \text{ and } r_{1}+r_{2}\ge 0
\]

\subsection{Transformation Rule for $\varphi$}
Next we consider somewhat more abstractly the process of cutting open a Haken manifold along a surface in its hierarchy.

Let $G^{3}$ be a cutting hypersurface in a Haken 4-manifold $X^{4}$ and let $Y^{4}=X^{4}|G^{3}$.  We assume that $G^{3}$ is connected, although $\partial G^{3}$ may well be disconnected. We know that $\chi(Y^{4})=\chi(X^{4})+\chi(G^{3})$. We need $\varphi(Y^{4})-\chi(G^{3})\ge \varphi(X^{4})$, or $\varphi(Y^{4}) \ge \varphi(X^{4}) + \chi(G^{3})$.
In this notation we know that a suitable $\varphi$ function must have the form
\[
\varphi(X^{4})=r_{0}f_{0}(X^{4}) +  r_{1}f_{1}(X^{4})+r_{2}f_{2}(X^{4})+r_{3}f_{3}(X^{4})+\frac14\chi(F^{3}X^{4}) -r_{3}b_{0}(\partial X^{4}).
\]

Now $F^{0}Y^{4} = F^{0}X^{4}\sqcup 2F^{0}G^{3}$. In words, the vertices of $Y^{4}$ consist of the vertices of $X^{4}$, together with two copies of the vertices of $G^{3}$. Note also that the vertices of $G^{3}$ are exactly where $\partial G^{3}$ meets $F^{1}X^{4}$.

Similarly $F^{1}Y^{4}=F^{1}X^{4}|F^{0}G$, together with two copies of $F^{1}G^{3}$, which is exactly $\partial G^{3}\cap F^{2}X^{4}$. In words the $1$-faces of $Y^{4}$ consist of the 1-faces of $X^{4}$ (as cut open by $\partial G^{3})$ together with two copies of the 1-faces of $G^{3}$.

Further  $F^{2}Y^{4}=F^{2}X^{4}|F^{1}G$, together with two copies of $F^{2}G^{3}$, which is exactly $\partial G^{3}\cap F^{3}X^{4}$.

Finally $F^{3}Y^{4}=F^{3}X^{4}|F^{2}G$, together with two copies of $G^{3}$.

Thus, using the sum formula for Euler characteristics,  we observe that
\begin{align*}
\chi (F^{0}Y^{4}) &= \chi (F^{0}X^{4}) +  2\chi (F^{0}G^{3})
\\
\chi (F^{1}Y^{4}) &= \chi (F^{1}X^{4}) +  \chi (F^{0}G^{3}) +2 \chi (F^{1}G^{3})
\\
\chi (F^{2}Y^{4}) &= \chi (F^{2}X^{4}) + \chi (F^{1}G^{3}) + 2\chi (F^{2}G^{3})
\\
\chi (F^{3}Y^{4}) &= \chi (F^{3}X^{4})  + \chi (F^{2}G^{3}) + 2\chi (G^{3})\\
\end{align*}
About the face numbers themselves we must be  less precise, noting that in general a connected codimension-one Haken manifold might or might not separate a connected manifold into two pieces. \begin{align*}
f_{0}(Y^{4})&= f_{0}(X^{4})+2f_{0}(G^{3}) \\
f_{1}(X^{4})+2f_{1}(G^{3}) \le f_{1}(Y^{4})&\le f_{1}(X^{4})+2f_{1}(G^{3})+f_{0}(G^{3})\\
f_{2}(X^{4})+2f_{2}(G^{3}) \le f_{2}(Y^{4})&\le f_{2}(X^{4})+2f_{2}(G^{3})+f_{1}(G^{3})\\
f_{3}(X^{4})+2b_{0}(G^{3}) \le f_{3}(Y^{4})&\le f_{3}(X^{4})+2b_{0}(G^{3})+f_{2}(G^{3})
\end{align*}

We also need
\[
b_{0}(\partial X^{4})\le b_{0}(\partial Y^{4})\le b_{0}(\partial X^{4})+{1}
\]
(assuming $G^{3}$ is connected).

We require
\[
\chi(X^{4})=\chi(Y^{4})-\chi(G^{3})\ge \varphi(Y^{4}) - \chi(G^{3}) \ge \varphi(X^{4}).
\]

From the above inequalities, taking into account the signs of the coefficients, we have
\begin{align*}
\varphi (Y^{4}) -\chi (G^{3})&\ge&\\
&\ge\varphi (X^{4}) +2r_{0}f_{0}(G^{3})+2r_{1}f_{1}(G^{3})+2r_{2}f_{2}(G^{3})+r_{2}f_{1}(G^{3})
+2r_{3}-r_{3}(b_{0} G^{3}+{1}) \\ 
&+ \frac{1}{4}\chi (F^{2}G^{3}) + \frac{1}{2}\chi (G^{3}) - \chi( G^{3})\\
&\ge \varphi (X^{4})
\end{align*}
Therefore we require
\[
2r_{0}f_{0}(G^{3})+2r_{1}f_{1}(G^{3})+2r_{2}f_{2}(G^{3})+r_{2}f_{1}(G^{3})
+2r_{3}-r_{3}(b_{0}(\partial G^{3})+{ 1}) + \frac{1}{4}\chi (F^{2}G^{3}) - \frac{1}{2}\chi (G^{3})
\ge 0
\]
 Next substitute $r_{3}=-r_{2}$ to obtain
 \[
2r_{0}f_{0}(G^{3})+2r_{1}f_{1}(G^{3})+r_{2}(2f_{2}(G^{3})+ f_{1}(G^{3})
-2+b_{0}(\partial G^{3})+{ 1}) + \frac{1}{4}\chi (F^{2}G^{3}) - \frac{1}{2}\chi (G^{3})
\ge 0
\]
We now apply this formula to the case of a Haken 3-manifold $G^{3}$ with essential boundary $\partial G^{3}=T_{g}$, a closed surface of genus $g$. In this case we have $f_{0}(G^{3})=f_{1}(G^{3})=0, f_{2} (G^{3})=1$, and $\chi (F^{2}G^{3})=2-2g$. Therefore
\[
r_{2}(2+0-2+1+{ 1} ) +\frac{1}{4}(2-2g)-\frac{1}{2}(1-g)\ge 0, \text{ or }  2r_{2}\ge 0
\]
Since we already know that $r_{2}\le 0$, we conclude that $r_{2}=0$ (and, then, $r_{3}=0$).{}

Thus our relation for Haken 3-manifolds becomes
\[
2r_{0}f_{0}(G^{3})+2r_{1}f_{1}(G^{3}) + \frac{1}{4}\chi (F^{2}G^{3}) - \frac{1}{2}\chi (G^{3})
\ge 0
\]

At this stage we may assume that the function $\varphi(X^{4})$ has the following form
\[
\varphi(X^{4})=r_{0}f_{0}(X^{4})+ r_{1}f_{1}(X^{4}) + \frac14\sum_{F^{3}<X^{4}}\chi(F^{3}),\text{ where } r_{1}\ge 0.
\]

If we apply the relation
\[
 r_{0}(2f_{0}(G^{3}))+ r_{1}(2f_{1}(G^{3}))
 + \frac{1}{4}\left(\chi (F^{2}G^{3})+\chi (\partial G^{3})  \right)
\ge \chi G^{3}
\]
to a Haken $3$-cell, we would have $\chi(G^{3})=1$, $\chi(\partial G^{3})=2$, $3f_{0}(G^{3})=2f_{1}(G^{3})$, and $\chi (F^{2}G^{3})=f_{2}(G^{3})$, yielding
\[
(2r_{0}+3r_{1})f_{0}(G^{3}) + \frac{1}{4}\left(f_{2}(G^{3}) + 2\right) \ge 1
\]
Also the relation $f_{0}-f_{1}+f_{2}=2$ yields $f_{2}=\frac{1}{2}f_{0}+2$. so that 
\[
(2r_{0}+3r_{1})f_{0}(G^{3}) + \frac{1}{4}\left(\frac{1}{2}f_{0}(G^{3}) + 4\right) \ge 1
\]
which implies the additional restriction
\[
2r_{0}+3r_{1} \ge -1/8.
\]

\subsection{Standard Haken cells}

For Haken $4$-cells, and, more generally, Haken manifolds in which all faces are Haken cells, all Euler characteristics of all faces are $+1$. In particular $\sum_{F^{k}<X^{4}}\chi(F^{k}) = f_{k}(X^{4})$, and this simplifies some aspects. 

At this stage our formula for $\varphi(X^{4})$ is
\[
\varphi(X^{4})=r_{0}f_{0}(X^{4})+ r_{1}f_{1}(X^{4}) + \frac14\sum_{F^{3}<X^{4}}\chi(F^{3}),\text{ where } r_{1}\ge 0, 2r_{0}+3r_{1} \ge -1/8.
\]
For a Haken $4$-cell we would have $f_{1}=2f_{0}$ and $\sum_{F^{3}<X^{4}}\chi(F^{3})=f_{3}(X^{4})$, and $\chi(X^{4})=1$. Hence for a Haken 4-cell the formula takes the shape
\[
\varphi(X^{4})=(r_{0}+2 r_{1})f_{0}(X^{4}) + \frac{1}{4}f_{3}(X^{4}),
\]
and we have the requirement that for a Haken 4-cell
\[
(r_{0}+2 r_{1})f_{0}(X^{4}) + \frac{1}{4}f_{3}(X^{4})\le 1.
\]

Plugging explicit Haken 4-cells into the formula gives another restriction on the coefficients. 
\begin{prop}
We must have $r_{0}+2r_{1}\le -1/16.$
\end{prop}
\begin{proof}
If $X^{4}=I^{4}$, with $f$-vector $(16, 32, 24, 8)$,  we have
\(
(r_{0}+2 r_{1})16+ \frac{1}{4}8\le 1.
\)
Hence
\(
r_{0}+2r_{1}\le -1/16.
\)
\end{proof}

 \subsection{Final determination of the $\varphi$-function}
\begin{lemma}
The system of inequalities 
\begin{align*}
r_{0}+2r_{1}&\le -1/16\\
2r_{0}+3r_{1} &\ge -1/8\\
r_{1} &\ge 0
\end{align*}
has unique solution $r_{0}=-1/16$, $r_{1}=0$.
\end{lemma}
\begin{proof}
The second inequality is equivalent to $-2r_{0}-3r_{1} \le 1/8$, the first inequality yields $2r_{0}+4r_{1} \le -1/8$, and adding we find $r_{1}\le 0$. We conclude that $r_{1}=0$. Then the first two inequalities show that $r_{0}\le -1/16$ and $r_{0}\ge -1/16$, so that $r_{0}=-1/16$. 
\end{proof}

In particular, there is a unique possible $\varphi$-function given as
\[
\varphi(X^{4})= -\frac{1}{16}f_{0}(X^{4})+\frac{1}{4}\sum_{F^{3}<X^{4}}\chi(F^{3}).
\]
This completes the proof of Theorem \ref{thm:possibleformula}.

\subsection{The transformation law}

Here we complete the proof of Proposition \ref{prop:transformationlaw}.

We have
\[
\varphi(X^{4})=-\frac{1}{16}f_{0}(X^{4}) + \frac{1}{4}\chi(F^{3}X^{4}).
\]
We suppose $X^{4}$ is a Haken $4$-manifold with $G^{3}\subset X^{4}$  a connected essential hypersurface, and set $Y^{4}=X^{4}| G^{3}$, with the induced boundary pattern. Then we need to prove that
\[
\varphi(Y^{4}) 
=  \varphi(X^{4}) -\frac18f_{0}(G^{3})  
 + \frac14\left(\chi (F^{2}G^{3})+\chi (\partial G^{3})  \right) .
\]

\begin{proof}
We have
\begin{align*}
\varphi(Y^{4})&=-\frac{1}{16}f_{0}(Y^{4}) +  \frac14\sum_{F^{3}<Y^{4}}\chi(F^{3}Y^{4})
\\
&= -\frac{1}{16}\left(f_{0}(X^{4})+2f_{0}(G^{3})\right) +  \frac14\left(\chi (F^{3}X^{4})  + \chi (F^{2}G^{3}) + 2\chi (G^{3})\right)
\\
&=
\varphi(X^{4})-\frac18f_{0}(G^{3}) +   \frac14\left( \chi (F^{2}G^{3}) + 2\chi (G^{3})\right)
\\
&=
\varphi(X^{4}) -  \frac{1}{8}f_{0}(G^{3})  + \frac14\left( \chi (F^{2}G^{3}) + \chi (\partial G^{3})\right).
\end{align*}
\end{proof}

\section{Evaluation of the $\varphi$-function on Haken 4-cells}\label{sec:duality}
Here we give a characterization of Haken 4-cells in terms of corresponding dual flag simplicial 3-spheres. This allows us to relate the evaluation of the $\varphi$-function to the Charney-Davis conjecture.
\subsection{Discussion of Dual Complexes and Flag Simplicial Spheres}

In polytope theory there is a basic duality in which every convex polytope $P$ has a dual polytope $P^{*}$ and for which simple polytopes are exactly the duals of simplicial polytopes. Moreover the double dual $P^{**}$ is combinatorially equivalent to $P$.  If $P$ has dimension $d$, then $P^{*}$ also has dimension $d$, and the face numbers satisfy $f_{k}P^{*} = f_{d-k-1}P$. A similar relation holds for more general forms of duality considered here. See Ziegler \cite{Ziegler1995}, for example.

For a general simplicial $n$-manifold $K$ and $p$-simplex $\sigma^{p}\in K$ the dual cone to $\sigma^{p}$ is a topological $(n-p)$-homology cell (a cone on a homology sphere)  that is realized as a subcomplex of the first barycentric subdivision $K'$ of $K$. Explicitly we can write
\[
D(\sigma)=\{\hat{\sigma} *\hat{\tau}_{1}*\hat{\tau}_{2}*\cdots *\hat{\tau}_{k}\}
\]
where $\sigma\subset \tau_{1}\subset \tau_{2}\subset \cdots \subset \tau_{k}$ is a chain of proper inclusions of simplices in $K$  containing $\sigma$. Here $\hat{\rho}$ denotes the barycenter of a simplex $\rho$. Up to homeomorphism $D(\sigma)$ is equivalent to the cone on the link of $\sigma$ in $K$. In dimensions greater than or equal to 5 this dual cone need not necessarily be a topological cell.

In order to guarantee that the dual cones are cells, the standard property to assume is that the triangulations in question are ``combinatorial'' or equivalently ``PL'', meaning that the links of simplices are PL equivalent to standard PL spheres of the appropriate dimension.  In lower dimensions, with which we are primarily concerned here (at most 4, boundary at most 3), the PL property is automatic. This is not true in dimensions greater than 4 because of the existence of non-combinatorial triangulations that arise from multiple suspensions of non-simply connected homology 3-spheres. It follows that the dual cones to simplices in a simplicial $n$-manifold, $n\le 4$,  are always cells.

In higher dimensions the triangulations we will obtain in general have the  possibly weaker property that links of all simplices are only \emph{homeomorphic} to spheres of the appropriate dimensions. 

We summarize the preceding discussion in the following proposition.

\begin{prop}
If $K$ is a simplicial $n$-manifold with the property that links of $k$-simplices are homeomorphic to $S^{n-k-1}$ for $k=0,1,\dots, n-1$, then the dual cones are topological cells and define a simple regular cell complex structure on the underlying space $|K|$. \qed
\end{prop}

We consider a regular cell complex structure on an $n$-sphere and the additional properties necessary for it to give a Haken cell structure. A \emph{regular} cell complex is one in which each closed cell is embedded and has its boundary a union of lower dimensional cells. A regular cell complex structure on an $n$-manifold is \emph{simple} if each $k$-cell is the intersection of exactly $n-k+1$ of the $n$-cells in the  cell complex.
\begin{prop}
Given a simple regular cell structure $X$ on an $n$-manifold $M^{n}$ there is a tame simplicial structure $K$ on $M^{n}$ such that $X$ is equivalent to the dual cell decomposition associated to $K$.
\end{prop}
\begin{proof}[Proof sketch]
In a manner similar to that used for a simplicial complex, the  regular cell complex  has a barycentric subdivision obtained inductively by starring in order of increasing dimension each cell at a point (which we refer to somewhat loosely as a barycenter) in its interior. This defines a simplicial complex $L$ with underlying space $M^{n}$ in such a way that the cells of the regular cell structure are obtained as appropriate unions of simplices of $L$. A typical $k$-simplex of $L$ has the form
\[
\hat{\tau}_{0} *\hat{\tau}_{1}*\hat{\tau}_{2}*\cdots *\hat{\tau}_{k}
\]
where
$\tau_{0}\subset \tau_{1}\subset \tau_{2}\subset \cdots \subset \tau_{k}$ is a chain of proper inclusions of cells in $X$.

Assuming that the given cell complex $X$ is simple, we argue that the simplices of $L$ can be regrouped to form another simplicial complex $K$ for which $L$ is the barycentric subdivision.

For a given vertex $x$ of $X$, which will also be a vertex of $L$, we consider the collection of all simplices of the form
\[
x *\hat{\tau}_{1}*\hat{\tau}_{2}*\cdots *\hat{\tau}_{n}
\]
and their faces, where
$\{x\}\subset \tau_{1}\subset \tau_{2}\subset \cdots \subset \tau_{k}$ is a chain of proper inclusions of cells in $X$. By simplicity of the original cell structure $X$ the union of these simplices defines an $n$-simplex 
\[
S_{x}:=\hat{\sigma}_{0}*\hat{\sigma}_{1}*\dots *\hat{\sigma}_{n}
\]
where $\sigma_{0}, \sigma_{1},\dots,\sigma_{n}$ are the $n+1$ $n$-cells of $X$ containing the vertex $x$.
The collection of these $n$-simplices $S_{x}$, $x$ a vertex of $X$, together with all their faces, defines a simplicial structure $K$ on $M^{n}$. Tracing through definitions we see that $L$ is exactly its barycentric subdivision.
\end{proof}

\begin{remark}
The preceding result is a mild topological extension of a result of M. Bayer \cite{Bayer1988}.
\end{remark}
\begin{prop}%
If $X^{n}$ is a Haken n-cell, then
\begin{enumerate}
\item
There is no 3-cycle in the 1-skeleton, and
\item
There is no empty triangle in the 1-skeleton of the dual simplicial $(n-1)$-sphere.
\end{enumerate}
\end{prop}
\begin{proof}
The assertion is clear by definition if $n\le 2$, where the two conditions essentially coincide, so assume $n\ge 3$.

Let $v_{0}, v_{1}, v_{2}$ be the vertices of a 3-cycle in the 1-skeleton of $X$. Each vertex is the intersection of exactly $n$ facets. Each edge is the intersection of exactly $n-1$ of the $n$ facets corresponding to either of its endpoints. It follows that all three vertices lie in exactly $n-2$ facets. Since $n\ge 3$ we conclude that the three vertices are the vertices of a triangular 2-face. But every face of a Haken $n$-cell (of any dimension) is also a Haken cell. Since a triangular face is not Haken, we have assertion (1).

Assertion (2) is part of being a ``useful'' boundary pattern (see section \ref{useful}).
\end{proof}

Recall that a simplicial complex in which any collection of $k+1$ pairwise adjacent vertices spans a $k$-simplex is called a \emph{flag simplicial complex}. Suggestively we think of a non-flag complex as having an \emph{empty simplex} of some dimension greater than 1, i.e., a subcomplex equivalent to the boundary of a $k$-simplex that does not span a $k$-simplex.
\begin{thm}\label{thm:simplicial}
The simplicial dual of the boundary complex of a Haken $4$-cell is a flag simplicial complex.
\end{thm}
\begin{proof}
By the preceding proposition there is no empty 2-simplex. Suppose there is an empty 3-simplex in the simplicial dual. Such an empty 3-simplex is a separating topological 2-sphere triangulated as $\partial \Delta^{3}$. Then the simplicial boundary sphere  can be expressed as a simplicial sum of two simplicial spheres. 
On the corresponding dual Haken $4$-cell side this corresponds to ``vertex sum'' of two simple regular  4-cells. As part of that vertex sum there will be faces that are vertex sums of two  simple, regular 3-cells. 
Such an object always has a forbidden essential small 2-disk, arising from the link of a vertex in one of the simple, regular 3-cells being summed. Therefore there can be no empty 3-simplices. Finally, if there were an empty 4-simplex, then the entire boundary pattern would be that of the boundary of a 4-simplex, which is not a Haken cell. 
\end{proof}

\subsection{Haken 4-cells and the Charney-Davis quantity}
With the information in hand that Haken 4-cells have boundary complex  dual to a flag simplicial 3-sphere, we verify that the only possible $\varphi$-function given by
\[
\varphi(X^{4}) = -\frac{1}{16}f_{0}(X^{4}) + \frac{1}{4}f_{3}(X^{4})
\]
satisfies the remaining necessary condition.
{}
\begin{thm} 
\label{thm:4cellinequality}If $X^{4}$ is a Haken 4-cell, then  $-\frac{1}{16}f_{0}(X^{4}) + \frac{1}{4}f_{3}(X^{4})\le 1$.
\end{thm}

\begin{proof} 
This is an interpretation of the Charney-Davis conjecture \cite{CharneyDavis1995a}, resolved affirmatively in dimension 3 by Davis-Okun \cite{DavisOkun2001}. The Charney-Davis conjecture for \emph{flag} triangulations of 3-dimensional spheres states that
\[
1 - \frac{1}{2}f_{0}^{*} + \frac{1}{4}f_{1}^{*} - \frac{1}{8}f_{2}^{*} + \frac{1}{16}f_{3}^{*} \ge 0
\]
where to avoid some confusion we use $f_{i}^{*}$ to denote the number of $i$-simplices.

Using the identities $f_{0}^{*} -f_{1}^{*} +f_{2}^{*} -f_{3}^{*} =0$ and $4f_{3}^{*} =2f_{2}^{*} $, we have $f_{3}^{*} =f_{1}^{*} -f_{0}^{*} $ and $f_{2}^{*} =2(f_{1}^{*} -f_{0}^{*} )$, and we see that the Charney-Davis inequality is equivalent to $ f_{1}^{*}  \ge 5f_{0}^{*}  - 16 $ (as compared to the known, standard lower-bound  inequality $f_1^{*} \ge 4f_0^{*}  - 10$ for general simplicial 3-spheres).

The Charney-Davis inequality above dualizes into $f_{2}\ge 5f_{3}-16$, which becomes the  inequality $f_0 \ge 4f_3 -16$ when we use the dualization of $f_{3}^{*} =f_{1}^{*} -f_{0}^{*} $ above into $f_{0}=f_{2}-f_{3}$. It follows that $-\frac{1}{16}f_{0}(X^{4}) + \frac{1}{4}f_{3}(X^{4})\le 1$, as required.
\end{proof}

\section{Concluding Discussion}
In subsequent work with M. Davis we have found a more conceptual and less computational analysis that applies in all even dimensions, showing that the Euler Characteristic Sign Conjecture for Haken manifolds reduces to the Charney-Davis conjecture in all odd dimensions. We also show that the duals of Haken $n$-cells yield flag simplicial spheres in all dimensions. Details will appear in a forthcoming paper. It still seems possible, however, that there are Haken $n$-manifolds that are counterexamples to the Euler Characteristic Sign Conjecture in higher dimensions.  And in particular we suggest the problem of constructing a closed Haken 6-manifold with positive Euler characteristic, such as a Haken 6-manifold with the rational homology  the 6-sphere or  complex projective 3-space.

\bibliography{../HakenManifolds}
\bibliographystyle{plain}
\end{document}